\documentclass{amsart}

\usepackage{amsmath}
\usepackage{amssymb}
\usepackage{latexsym}
\usepackage{amsfonts}
\usepackage{mathrsfs}
\usepackage{amsthm}
\usepackage{hyperref}
\usepackage{cases}

\newcommand{\OO}{\mathcal {O}}
\newcommand{\F}{\mathbf {F}}
\newcommand{\A}{\mathbb {A}}
\newcommand{\N}{\mathbb {N}}

\newcommand{\eps}{\varepsilon}

\DeclareMathOperator{\Spec}{Spec}
\DeclareMathOperator{\edim}{edim}
\DeclareMathOperator{\ecodim}{ecodim}
\DeclareMathOperator{\ejump}{ejump}

\DeclareMathOperator{\pdeg}{\mathit{p}-deg}
\DeclareMathOperator{\trdeg}{tr.deg}

\newtheorem{theorem}[equation]{Theorem}

\newtheorem{lemma}[equation]{Lemma}
\newtheorem{proposition}[equation]{Proposition}

\newtheorem{corollary}[equation]{Corollary}
\newtheorem{remark}[equation]{Remark}
\newtheorem{notation}[equation]{Notation}

\newtheorem{question}[equation]{Question}

\theoremstyle{definition}
\newtheorem{definition}[equation]{Definition}
\newtheorem{definition/lemma}[equation]{Definition/Lemma}

\input xy
\xyoption{all}

\title{A bound on embedding dimensions of geometric generic fibers}
\author{Zachary Maddock}

\begin{document}
\maketitle
\begin{abstract}
  We limit the singularities that arise in geometric generic
  fibers of morphisms between smooth varieties of positive
  characteristic by studying changes in embedding dimension 
  under inseparable field extensions.
  We then use
  this result in the 
  context of the minimal model program to
  rule out the existence of smooth varieties fibered by certain
  non-normal del Pezzo surfaces over bases of small dimension.
\end{abstract}

\section{Introduction}

This paper investigates the singularities that arise in generic
fibers of morphisms between smooth varieties in positive
characteristic.  In characteristic $0$, any morphism
between smooth varieties admits a dense open locus of the
base
 over which all fibers are smooth.
  However, over 
fields of positive characteristic this is no longer the
case, as there exist morphisms between smooth varieties in which 
every fiber is singular (that is, non-smooth over its base field).  
A simple example,  occurring over an arbitrary field $k$ of characteristic $2$
(resp.~$3$), is the morphism $f:\A^2_k \to \A^1_k$ given by 
$(x,y) \mapsto x^2 + y^3$.  
The fiber of
$f$ over
any point $t_0 \in \A^1_k$ is the planar curve defined by the equation
$x^2 + y^3 - t_0 = 0$, which clearly has a cuspidal singularity at the
geometric point $(\sqrt{t_0}, 0)$ (resp.~$(0, \sqrt[3]{t_0})$).

This phenomenon is more than just pathology, rather it is a feature 
of positive characteristic geometry that arises naturally when
attempting 
to study a class of smooth varieties
via morphisms to other
varieties.  One instance of this occurs in Mumford and Bombieri's
classification
of fibrations in characteristic $p > 0$, within the context of the
Enriques classification of surfaces (cf.~\cite{bom-mum-III},\cite{bom-mum-II}).
As the above
example illustrates, when $p= 2$ or
$3$ there exist smooth surfaces fibered
in cuspidal curves of arithmetic genus $1$
(i.e.~the \emph{quasi-elliptic fibrations}).

\subsection*{Main results}
 Non-smooth points 
 in the generic fiber 
of a
morphism
lie under those 
in the \emph{geometric} generic fiber, an algebraically closed
field extension of the generic fiber, 
at which
the stalk of the structure sheaf fails to be a regular local ring (cf.~Def.~\ref{definition-of-smooth}).  To
measure this failure, it is useful to recall the following definition:

\begin{definition}
The difference by which the embedding dimension
(cf.~\S\ref{section-defs}) at a (possibly 
non-closed) point $z$ of a 
variety $Z$ exceeds the codimension of that point is called 
the \emph{embedding codimension} of $z$ in $Z$,
$$\ecodim_Z(z) := \edim(\OO_{Z,z}) - \dim(\OO_{Z,z}).$$
\end{definition}

Clearly the embedding codimension of $z \in Z$ is nonnegative,
and equals zero if and only 
if $\OO_{Z,z}$ is a regular local ring, so we see that it does provide
some measure of the singularity at $z$. 
The main result of this paper
is the following bound on the embedding codimension of points in the
geometric generic fiber of a morphism between smooth varieties, 
which thus limits the possible singularities that arise:\\ 

\vspace*{-1em}

\begin{theorem}\label{theorem-intro-main}
Let $f: X \to S$ be a morphism between smooth varieties over a perfect
field $k$.  Then the generic fiber $X_{\xi}$ is a regular variety over the
function field of $S$, $\kappa := \kappa(\xi)$, and 
 any point $\bar x$ in the geometric generic fiber 
$\overline{X}_{\xi} := X_{\xi}
\times_{\kappa} \bar \kappa$ satisfies
\begin{equation}\label{inequality-geometric-version}
\ecodim_{\overline{X}_{\xi}}(\bar x) \leq \dim(S).
\end{equation}
\end{theorem}

\begin{remark} 
The bound on embedding
codimension asserted in Theorem~\ref{theorem-intro-main}
is immediate for \emph{special} fibers.  This is because the geometric
fiber $\overline{X}_s := X_s \times_{\kappa(s)} \bar k$ over any
closed
point $s \in S$ embeds via a closed
immersion into the
smooth variety  $\overline{X} := X \times_k \bar k$, so it follows
that 
 $\edim_{\overline{X}_s}(\bar x) \leq \edim_{\overline{X}}(\bar x)$
 for all $\bar x \in \overline{X}_s$, and
 consequently that
\begin{align*}
\ecodim_{\overline{X}_s}(\bar x) & = \edim_{\overline{X}_s}(\bar x)
 - \dim(\OO_{\overline{X}_s,\bar x})\\
                 & \leq \edim_{\overline{X}}(\bar x)
                 - \dim(\OO_{\overline{X}_s,\bar x})\\ 
& = \dim(\OO_{\overline{X},\bar x}) -  \dim(\OO_{\overline{X}_s,\bar x})\\
& = \dim(\overline{X}) - \dim(\overline{X}_s)\\
& \leq \dim(S).
\end{align*}
%
The content of the theorem is
 that
this inequality, which easily holds for all special fibers, also
 holds for the generic fiber.  
\end{remark}

\subsection*{Main application}

Our primary
application of the above theorem is in the setting of the minimal
model program, where one studies a higher-dimensional 
variety via its morphisms to simpler varieties.
A primary goal in the program is to construct,  from a given variety 
$X$,
a minimal model by contracting each extremal curve $C
\subseteq X$ that pairs 
negatively with the canonical divisor in $X$. 
If the curve $C$ is
sufficiently mobile in $X$, then this contraction morphism may not be
birational, and 
instead may be a fibration by Fano schemes.

In positive characteristic, Koll\'ar
demonstrated the
existence of 
these contraction morphisms on smooth $3$-folds $X$, extending a result
of Mori from characteristic $0$ (cf.~\cite{kol0},\cite{mor}). 
Furthermore, he gives a detailed classification of the geometry of the
possible contractions $f: X \to X'$ in the case where $f$ is
birational (i.e.~when $X'$ is a $3$-fold).
If $X'$ is a surface then
$f$ is simply a conic bundle, but 
if $X'$ is a curve then $f$ is a fibration
by del Pezzo surface schemes, and Koll\'ar remarks that the geometry
here could potentially be rather complicated.  He raises the
question of whether 
the geometric generic fibers of such $f$ can be non-normal
(cf.~\cite[Rem.~1.2]{kol0}) 
and if so, could the the generic fiber $Y$ of a del Pezzo surface
fibration satisfy $H^1(Y, \OO_Y) \neq 0$
(cf.~\cite[Rem.~5.7.1]{kol1}).
Over a perfect field, all normal del Pezzo surfaces $Y$ satisfy
$H^1(Y,\OO_Y) =0$ by a result of Hidaka and Watanabe
(cf.~\cite[Cor.~2.5]{hid-wat1}), 
although in positive characteristic $p > 0$, 
Reid exhibits 
non-normal del Pezzo surfaces  
 $Y$ with $H^1(Y,\OO_Y) \neq 0$ (cf.~\cite[\S4.4]{rei1}).

The author recently
constructed two projective morphisms $f: X \to S$ 
between smooth varieties of characteristic $2$ whose generic
fibers are 
regular del 
Pezzo surfaces $Y$ with $h^1(Y,\OO_Y) = 1$ (cf.~\cite{maddock-del}).
In one example, 
 $X$ is a $5$-fold\footnote{
It is actually possible to
create a similar example with $X$ a $4$-fold; the
details of which shall be included in a forthcoming paper.}
 and the geometric
generic fiber is integral but non-normal. 
  In the other example, $X$ is a
 $6$-fold and the geometric generic fiber is
  non-reduced.
It remains an open question whether del Pezzo surfaces $Y$ with
$H^1(Y, \OO_Y) \neq 0$ can arise as the
generic fiber of a morphism from a smooth $3$-fold to a curve, but it
follows from the 
main result of this paper that, at least in characteristics greater
than $3$, such geometry is not possible:\\

\begin{corollary}
Let $f: X \to C$ be a surjective morphism between a smooth $3$-fold $X$ and a
curve $C$ over a perfect
field of characteristic $p > 3$.  If the generic fiber $Y$ is a del
Pezzo surface (i.e.~if $\omega_Y^{-1}$ is ample), then $H^1(Y, \OO_Y) = 0$.
\end{corollary}

\subsection*{Connections to the literature}

Our main theorem is related to one of Schr\"oer
(cf.~\cite[Cor.~2.4]{sch-on}) 
which asserts that, in the case of a proper fibration  $f:
X \to S$, the inequality
\eqref{inequality-geometric-version} is strict if
$x \in \overline{X}_{\xi}$ is the generic point:

\begin{theorem}[Schr\"oer]
Let $f: X \to S$ be a proper morphism between integral
normal algebraic 
$k$-schemes of positive dimension satisfying $f_\ast(\OO_X) = \OO_S$, and
let $\xi \in S$ denote the generic point.  Then the geometric generic
embedding dimension of $X_{\xi}$ (i.e.~the embedding codimension of
the generic point of $\overline{X}_{\xi}$) is strictly less than $\dim(S)$.
\end{theorem}

  In the same work, Schr\"oer observes that a
$k$-scheme $X$ is geometrically reduced (i.e. $X_{\bar k}$ is reduced)
if and only if the base change $X_{k^{1/p}}$ of $X$ by the height $1$ 
field
extension $k \subseteq k^{1/p}$ is reduced.  The analogous property
for geometric regularity is a well-known result of EGA
(cf.~\cite[Thm.~IV.0.22.5.8]{EGA}).  We refine 
this result by proving:\\

\begin{proposition}
 Let $k$ denote a field of
characteristic $p > 0$ and let $x \in X$ denote a point in a $k$-variety 
 $X$.  If $x' \in X_{k^{1/p}}$ and
 $x^{(\infty)} \in X_{k^{1/p^\infty}}$ denote the
 preimages of $x$ under 
 the natural bijections $X_{k^{1/p^\infty}} \to X_{k^{1/p}} \to X$,
then 
$$\edim
_{X_{k^{1/p}}}
(x') = \edim
_{X_{k^{1/p^\infty}}}
(x^{(\infty)}).$$
\end{proposition}

\section{Regularity and smoothness}
\label{section-defs}
\setcounter{equation}{0}


We briefly recall the definitions of
the notions of regularity and smoothness.

\begin{definition}
  The \emph{embedding dimension} of a locally Noetherian scheme $X$ at
  a point $x \in X$ is
  the embedding dimension  of the local ring
  $\OO_{X,x}$ at the maximal ideal $\mathfrak m_x$, that is,
  the dimension of
  the Zariski cotangent space over the residue
  field $\kappa(x):= \OO_{X,x}/\mathfrak
  m_x$,
  $$\edim_X(x) = \edim(\OO_{X,x}) = \dim_{\kappa(x)} \mathfrak m_x/
  \mathfrak m_x^2.$$
\end{definition}
\begin{definition}
  A scheme $X$ is \emph{regular} if it is locally
  Noetherian and for all $x \in X$, the local
  ring $\OO_{X,x}$ is a regular local ring
 (i.e.~the embedding dimension of $\OO_{X,x}$
  is equal to its Krull dimension).
\end{definition}
\begin{definition}\label{definition-of-smooth}
  A scheme $X$ is
  \emph{smooth} over a field $k$ if it is locally
  of finite-type and 
  geometrically regular over $k$ (i.e.~$X \times_k \bar k$ is
  regular). 
\end{definition}
\begin{remark}
  Any smooth scheme is regular, and any
  regular scheme is locally integral.  Therefore, any connected,
  separated scheme of finite type over $k$ that is smooth over $k$ (or
  regular) is automatically a \emph{variety over $k$}, which for us
  refers to an
  integral, separated scheme of finite type over $k$. 

\end{remark}

Over a perfect field, the notions of regularity and smoothness are
equivalent.
However, over imperfect fields (of positive characteristic), a scheme
may be regular but not smooth.  We 
have already seen such an example: the generic fiber of the morphism
$f:\A^2_k \to \A^1_k$ described in the introduction
is regular at all points but is not
smooth since a cuspidal singularity appears after an algebraic
extension of the function field $k(t)$.
It turns out that pretty much all examples of
regular 
 varieties arise in this way, as generic fibers of 
morphisms between smooth varieties, and therefore the study of the
singularites appearing in the geometric 
generic fibers of morphisms between smooth varieties reduces to
the study of the singularities appearing in the geometric
(i.e.~algebraically closed) base changes
of regular varieties:

%
%

\begin{proposition}\label{prop-regular-is-generic-fiber}
  Let $Y$ be a variety over a finitely generated field extension $K$ of a
  perfect field $k$. 
  $Y$ is regular if and only if there
  exists a morphism of smooth $k$-varieties $f: X \to B$ so that 
  $K$ is the function field of $B$ and $Y$ is the generic fiber
  of $f$.
\end{proposition}
\begin{proof}
  See \cite[Prop.~1.6]{sch2}.
\end{proof}

Notice that if $X$ is a regular variety that is not smooth over $k$, then 
then there exists a closed
point $\bar x \in \overline X:= X \times_k \bar k$ sitting over some point
$x \in X$
such that 
\begin{align*}
\edim_{\overline{X}}(\bar x)  > \dim(\overline{X})
& = \dim (X) \\
 & = \edim_{X}(x).
\end{align*}
In this way, the existence of
regular but non-smooth schemes is directly linked to 
``jumps'' in embedding dimension after a geometric extension of
scalars $\bar k /k$.  


\section{Jumps in embedding dimension}
\label{section-jumps}
\setcounter{equation}{0}

For any purely inseparable field extension $k'/k$, the morphism of
affine schemes $\Spec k' \to \Spec k$ is a universal homeomorphism.
In particular, any $k$-algebra $R$ is local if and only if $R
\otimes_k k'$ is so, which shows the following definition to be
well-formed: 

\begin{definition}
  Let $R$ be a local Noetherian $k$-algebra and  $k'/k$
  a purely inseparable field extension.  We define the
  \emph{embedding jump} of $R$ over the extension $k'/k$ to be the difference
  between the embedding dimensions
  $$\ejump_{k'/k}(R) := \edim(R\otimes_k k') - \edim(R).$$ 
  The \emph{embedding jump} $\ejump_{k'/k}(x)$ of a scheme $X$ at a point $x
  \in X$ is defined by $\ejump_{k'/k}(x) := \ejump_{k'/k}(
  \OO_{X,x})$. 
\end{definition}

%
%
%

\begin{remark}\label{remark-jumps-are-nonnegative}
We make two easy observations about embedding jumps:
\begin{enumerate}
\item Embedding jumps are non-negative
  (cf.~\cite[0.IV.22.5.2.1]{EGA}).
\item Because $R$ is Noetherian, any purely inseparable
  field extension $k'/k$ admits some finite
  sub-extension $k \subseteq k''$ for which
  $\ejump_{k'/k}(R) = \ejump_{k''/k}(R)$.
\end{enumerate}
\end{remark}

In the special case that $X = \Spec K = \{x \}$, for a field extension 
$K/k$, the embedding dimension $\edim_X(x)$ is zero and so
the embedding jump is simply the embedding dimension of the Artin
local ring $K \otimes_k k'$,
$$\ejump_{k'/k}(x) = \edim(K \otimes_k k').$$
This quantity was studied by Schr\"oer in
\cite[Prop.~2.1]{sch-on}, where he proved the following theorem,
which implies the $0$-dimensional case of our main
result (cf.~Thm.~\ref{theorem-max-embedding-dim}):

\begin{proposition}[Schr\"oer]\label{proposition-schroeer}
Let $K/k$ be an extension of fields of  characteristic $p > 0$, and
let $k'/k$ be a field extension that contains 
$k^{1/p}$.  Then the embedding dimension 
of $K \otimes_k k'$ equals that of $K \otimes_k k^{1/p}$, which also
equals the difference between 
the $p$-degree and the transcendence degree of the field extension
$K/k$. 
\end{proposition}

For the reader's convenience, we next recall the definition of $p$-degree.

\section{The $p$-degree of a field extension}
\label{p-degree}
\setcounter{equation}{0}

 The sheaf of K\"ahler differentials $\Omega_{X/k}$ on
a variety $X$ is a locally free $\OO_X$-module of rank equal to $\dim X$ if
and only if $X$ 
is smooth over $k$.  In characteristic $0$, the transcendence degree
of a finitely 
generated field extension $K/k$ is equal to the rank of the $K$-vector
space of K\"ahler differentials $\Omega_{K/k}$.  In characteristic
$p$, this is no longer the case,
suggesting
that transcendence degree is perhaps not best-suited for
discussions of smoothness.

\begin{definition}
Let $K/k$ be an extension of fields of
characteristic $p > 0$.   The
\emph{$p$-degree of $K/k$} is defined to
be the rank of the $K$-vector space $\Omega_{K/k}$.
\end{definition}

\begin{remark}
If $K/k$ is an arbitrary
extension of fields of characteristic $p > 0$, then
$$\pdeg(K/k) = \pdeg(K/k(K^p)),$$
where $k(K^p)$ denotes the
subfield of $K$ generated by $k$ and $K^p$.  This is because
$\Omega_{K/k} 
= \Omega_{K/k(K^p)}$, which holds since
$d(f^p) = pf^{p-1}df = 0$ for all $f \in K$. 
\end{remark}

For a finitely generated extension $K$ of a perfect field $\F$, the
notion of $p$-degree and transcendence degree actually agree,
due to the existence of a separating transcendence basis (cf.~\cite[Thms.~26.2-3]{mat-com-ring}).

\begin{proposition}\label{proposition-pdegree-over-perfect-field}
If $\F$ is a perfect field and $K$ is a finitely generated field
extension, then
$$\pdeg(K/\F) = \trdeg(K/\F).$$
\end{proposition}
%
%

\section{Embedding jumps and residue fields}
\setcounter{equation}{0}
In this section we prove our main lemma that bounds the jump in
embedding dimension of a regular local noetherian ring by that of its
residue field.

\begin{lemma}\label{lemma-ejump-bound}
  Let $R$ be a local Noetherian ring, over a field $k$, with maximal
  ideal $\mathfrak m$ and residue
  field $\kappa = R / \mathfrak m$.  If $k'/k$ is a purely inseparable
  field extension,
  then
  $$\ejump_{k'/k}(R) \leq \edim(\kappa \otimes_k k').$$
\end{lemma}
\begin{proof}
  Denote by $R'$ the local ring $R \otimes_k k'$, by $\mathfrak m'$
  its maximal ideal, and by
  $\kappa'$ its residue field $R'/\mathfrak m'$.
  Consider the short exact sequence of $\kappa'$-vector spaces,
  \begin{equation}\label{equation-short-exact-sequence}
    0 \to \mathfrak m R' / ( \mathfrak m R' \cap {\mathfrak m'}^2) \to 
    \mathfrak m'/ {\mathfrak m'}^2 \to \mathfrak m'/(\mathfrak m R' +
    {\mathfrak m'}^2) \to 0.
  \end{equation}
  As a $\kappa'$-vector space, 
  the dimension of the middle term is $\edim(R')$, by definition. 
  
  We next consider the right-hand term, first noting the isomorphism
  \begin{equation*}
    (\mathfrak m'/ \mathfrak m R') \otimes_{R'} \kappa'
     \cong \mathfrak m' / (\mathfrak m R' + {\mathfrak m'}^2).
  \end{equation*}
  Clearly $\mathfrak m'/\mathfrak m R'$ is the maximal ideal of $R'/
  \mathfrak m R' \cong \kappa \otimes_k k'.$  Since $\kappa'$ is its residue
  field, 
  the $\kappa'$-dimension of $(\mathfrak m'/\mathfrak m R') \otimes_{R'} \kappa'$
  is equal to $\edim(\kappa \otimes_k  k')$, which therefore equals the
  dimension of the right-hand term of
  \eqref{equation-short-exact-sequence}.

  To analyze the left-hand term of
  \eqref{equation-short-exact-sequence},
  observe that 
  $$ \mathfrak m R'/ \mathfrak m  \mathfrak m' \cong (\mathfrak m /
  \mathfrak m ^2) \otimes_\kappa \kappa',$$ 
  and therefore
  $$\dim_{\kappa'}(\mathfrak m R'/\mathfrak m\mathfrak m') =
  \dim_\kappa (\mathfrak m/\mathfrak m^2) = \edim(R).$$
  Because of the natural inclusion, 
  $\mathfrak m \mathfrak
  m' \subseteq \mathfrak m R' \cap {\mathfrak m'}^2$, 
  we have the inequality
  $$\dim_{\kappa'}(\mathfrak m R'/(\mathfrak m R' \cap {\mathfrak m'}^2))
  \leq \dim_{\kappa'}(\mathfrak mR'/\mathfrak m \mathfrak m').$$
  
  From the short exact sequence \eqref{equation-short-exact-sequence},
  it then follows that
  \begin{align}\label{equation-dimension-inequality}
    \edim(R')& = \edim(\kappa\otimes_k k') + \dim_{\kappa'}(\mathfrak m R' /
    (\mathfrak m R' \cap {\mathfrak m'}^2)) \nonumber\\
    & \leq
    \edim(\kappa \otimes_k k') + \edim(R).\nonumber
\end{align}

\end{proof}



\section{A bound on embedding jumps}
\setcounter{equation}{0}
We now combine our main lemma with a result of Schr\"oer
to obtain a bound on the embedding jump at an arbitrary point of a regular
variety in terms of the $p$-degree and 
transcendence degree of the residue field at that point.

\begin{theorem}\label{theorem-max-embedding-dim}
  Let $X$ be a regular $k$-variety.
  If $k'/k$ is a purely inseparable extension, then for any $x \in X$
  with residue field $\kappa(x)$,
  the embedding jump  satisfies
  \begin{equation*}
    \ejump_{k'/k}(x)  \leq p\textrm{-}\mathrm{deg}(\kappa(x)/k) -
    \mathrm{tr.deg}(\kappa(x)/k).
  \end{equation*}
\end{theorem}
\begin{proof}
  By Lemma \ref{lemma-ejump-bound}, the jump in embedding
  dimensions is bounded by
  $$\ejump_{k'/k}(x) \leq \edim(\kappa(x) \otimes_k
  k').$$
  Let $k''$ denote the subfield of the algebraic closure $\bar k$
  generated by $k'$ and 
  $k^{1/p}$.  By Remark \ref{remark-jumps-are-nonnegative},
  $\edim(\kappa(x)\otimes_k k') \leq \edim(\kappa(x) \otimes_k k'')$.
  Schr\"oer's result (Prop.~\ref{proposition-schroeer})
  implies that
  $\edim(\kappa(x) \otimes_k k'')$ equals $\edim(\kappa(x) \otimes_k
  k^{1/p})$
  and 
  also equals the difference between the $p$-degree and the transcendence
  degree of the extension $\kappa(x) / k$.  
\end{proof}

Our primary applications of the above result will be through the
following geometric consequence: 

\begin{corollary}\label{corollary-generic-fiber-embedding-dimension}
  Let $f: \mathcal X \to B$ be a morphism of smooth varieties over a
  perfect field $\F$.  Then the generic fiber $X$ is a 
  regular variety over the fraction field $k$ of $B$.  Moreover, the
  embedding dimension $\edim_{\overline{X}}(\bar x)$ at any point 
  $\bar x \in \overline{X} := X \times_k \bar k$ satisfies
  $$\edim_{\overline{X}}(\bar x) \leq \edim_X(x) + \dim(B),$$
  where $x \in X$ denotes the point lying under $\bar x \in \overline{X}$.
\end{corollary}
\begin{proof}
  By Theorem \ref{theorem-max-embedding-dim},
  $\ejump_{\bar k/ k}(x) \leq \pdeg(\kappa(x)/k) -
  \trdeg(\kappa(x)/k)$, 
  where $\kappa(x)$ denotes the residue field of $x \in X$.
  Clearly 
  $$\pdeg(\kappa(x)/k) \leq \pdeg(\kappa(x)/\F) =
  \trdeg(\kappa(x)/\F),$$
  with the latter equality following from Proposition
  \ref{proposition-pdegree-over-perfect-field}. 
  Therefore,
  \begin{align*}
    \edim_{\overline{X}}(\bar x) - \edim_X(x) 
    & \leq
    \trdeg(\kappa(x)/\F) - \mathrm{tr.deg}(\kappa(x)/k)\\
    &  = \trdeg(k/\F)\\
    & = \dim(B).
  \end{align*}
\end{proof}

\section{Regular del Pezzo surfaces}
\setcounter{equation}{0}

The primary motivation for this investigation was to determine which
singular del Pezzo surfaces can occur as the geometric generic fiber of the
contraction of an extremal curve class on a smooth $3$-fold.  Although
we do not answer this question definitively, the above 
results do rule out the nasty examples in characteristics $p >
3$ of
 non-normal del Pezzo surfaces $X$
with $H^1(X,\OO_X) \neq 0$.

\begin{proposition}\label{proposition-h1=0-in-p>3}
  Let $X$ be a regular del Pezzo surface over a finitely generated
  field extension $k/\F$ of a perfect field $\F$ of
  characteristic $p$ and transcendence degree  $\trdeg(k/\F)= d$.
  If $d \leq 1$ then $X$ is geometrically reduced.  
  If $p > d + 2$ and $X$ is geometrically reduced,
  then $H^1(X, \OO_X) = 0$.
\end{proposition}
\begin{proof}
  If $k$ is of transcendence degree at most $1$ over the perfect field
  $\F$, then
  $\overline{X}  := X \times_k {\bar k}$ is reduced
  (cf.~\cite{sch-on}).
  By the classification of normal del Pezzo surfaces over an
  algebraically closed field (cf.~\cite{hid-wat1}), the
  result is true if $\overline{X}$ is normal.  
  This just leaves the case where $\overline{X}$ is integral but
  non-normal (and hence where $d > 0$).  Such 
  examples were classified by Reid (cf.~\cite{rei1}).
  In particular, in characteristics $p > 3$, the nonvanishing
  $H^1(\overline{X},\OO_{\overline{X}}) \neq 0$ is only
  possible when there exists points $ \bar x \in \overline{X}$ with
  $\edim_{\overline{X}}(\bar x) = p$.
  (cf.~\cite[\S4.4]{rei1}).   
  By  Corollary \ref{corollary-generic-fiber-embedding-dimension},
  $\edim_{\overline{X}}(\bar x) \leq d + 2$ for all $\bar
  x \in \overline{X}$, and 
  therefore $H^1(\overline{X}, \OO_{\overline{X}}) = 0$, which 
  implies $H^1(X, \OO_X) = 0$.  
\end{proof}

\section{Jumping is a height one phenomenon}
\setcounter{equation}{0}

An extension of characteristic $p$ fields $L/K$ is said to be of
\emph{height} one if $L^p \subseteq K$. 
As a consequence of Theorem~\ref{theorem-max-embedding-dim}, we show
that jumps in embedding 
dimension are a strictly height one phenomenon.  As a corollary,
we recover the well-known result \cite[Thm.~IV.0.22.5.8]{EGA} 
that asserts that geometric regularity may be checked over
height one field extensions.
We set the following notation for this section:
\begin{notation}
  For an imperfect field $k$ of characteristic $p$, an element
  $t \in k \setminus 
  k^p$, and a $k$-algebra $R$, set
\begin{itemize}
\item $k_n := k(\sqrt[p^n]{t})$ and
\item $R_n := R \otimes_k k_n$.
\end{itemize}
\end{notation}

\begin{lemma}\label{lemma-roots-of-t}
    Let $R = K$ be a finitely generated field extension of an
    imperfect field $k$ of characteristic $p$.
    If $t \in k\setminus k^p$ and $m := \max \{k \in \N: t \in
    K^{p^k}\}$,
        then for all $0 < n \leq m$, the ring
    $$R_n  \cong 
    \begin{cases}
      K[\eps_n]/(\eps_n^{p^n}) & : \textrm{ if }0 \leq n \leq m\\
      K(\sqrt[p^n]{t})[\eps_n]/(\eps_n^{p^m}) & : \textrm{ if } m < n, 
    \end{cases}
    $$
    and the natural ring inclusion $R_{n-1} \subseteq R_{n}$ is
    given by 
$$    \begin{cases}
      \eps_{n-1} 
    \mapsto \eps_{n}^{p}& : \textrm{ if } 0 \leq n \leq m\\
    \eps_{n-1} \mapsto \eps_{n} & : \textrm{ if } m < n.
    \end{cases}$$
    In particular, the residue field of $R_n$ equals $K$ if and only
    $\sqrt[p^n]{t} \in K$. 
\end{lemma}
\begin{proof}
  If $0 \leq n \leq m$, then $\sqrt[p^n]{t} \in K$ and
  therefore $R_n = K \otimes_k k(\sqrt[p^n]{t})$ is isomorphic to the
  Artin local 
  ring $K[\eps_n]/(\eps_n^{p^n})$,
  where $\eps_n:= \sqrt[p^n]{t} \otimes 1 - 1 \otimes
  \sqrt[p^n]{t}$.
  Moreover, it is follows then that $\eps_{n-1} = \eps_{n}^p$.
  On the other hand, if $m < n$, then $\sqrt[p^n]{t} \notin R_{n-1}$.
  Moreover, since $k_m \subseteq K$, the result follows by composing the
  following isomorphisms:
  \begin{align*}
    K \otimes_k k(\sqrt[p^n]{t}) &\cong (K \otimes_k k_m) \otimes_{k_m}
    k_n\\
    & \cong K[\eps_m]/(\eps_m^{p^m}) \otimes_{k_m} k_m(\sqrt[p^n]{t}) \\
    & \cong K(\sqrt[p^n]{t})[\eps_m]/(\eps_m^{p^m}).
  \end{align*}
\end{proof}

\begin{proposition}\label{proposition-embedding-jump}
Let $k$ be an imperfect field of characteristic $p$.  If $t \in k\setminus
k^p$ and $R$ is a noetherian local $k$-algebra,
 then for any $n \geq 1$,
$$\ejump_{k_1/k}(R) = \ejump_{k_n/k}(R).$$
\end{proposition}
\begin{proof}
First, assume we have proven the result in the base case $n = 2$.
By applying this to the field $k_{n-2}$ and the noetherian
local $k_{n-2}$-algebra $R_{n-2}$, it would follow that for
any $n \geq 2$,  
$$\ejump_{k_n/k_{n-2}}(R_{n-2}) = \ejump_{k_{n-1}/k_{n-2}}(R_{n-2}).$$
Using this equality, we derive the general result by observing
\begin{align*}
  \ejump_{k_n/k}(R) & = \ejump_{k_n/k_{n-2}}(R_{n-2}) +
  \ejump_{k_{n-2}/k}(R)\\
  & = \ejump_{k_{n-1}/k_{n-2}}(R_{n-2}) +
    \ejump_{k_{n-2}/k}(R)\\
    & = \ejump_{k_{n-1}/k}(R),
\end{align*}
and then arguing inductively.
Thus, it suffices to prove the result in the case $n =2$.

Let $n =2$ and note by Remark~\ref{remark-jumps-are-nonnegative}(1)
and Theorem~\ref{theorem-max-embedding-dim} that
$$0 \leq \ejump_{k_1/k}(R) \leq \ejump_{k_2/k}(R) \leq 1.$$  
Equality follows immediately
in the case $\ejump_{k_1/k}(R) = 1$, so we henceforth 
assume $\ejump_{k_1/k}(R) = 0$.  It easily
follows that $\ejump_{k_2/k}(R) 
= \ejump_{k_2/k_1}(R_1)$, and we finish the proof by showing that this
quantity also is zero. 

Let $K, K_1,$ and $K_2$ be the residue fields
of $R$, $R_1$, and $R_2$, respectively, and denote by
$\mathfrak m$, $\mathfrak m_1$, and $\mathfrak m_2$ the corresponding
maximal ideals.  Clearly there are inclusions $K \subseteq K_1
\subseteq K_2$.  As  $K_1$ is a quotient of $K \otimes_k k_1$ and
$K_2$ is a quotient of $K_1 \otimes_{k_1} k_2$, 
it follows that $[K_1:K], [K_2:K_1] \in \{1,p\}$.
Moreover, $[K_1:K] = p$ if and only if $K_1 = K \otimes_k k_1$, that
is, if and only if $\mathfrak m_1 = \mathfrak m
R_1$, and   similarly, $[K_2:K_1] = p$ if and only if $\mathfrak m_2 =
\mathfrak m_1 R_2$.
We shall
conclude the proof by analyzing separately the
following cases:\\

\noindent\textbf{Case:}
$[K_2:K_1] = p$.\\
As noted above, this holds  only if $\mathfrak m_2 =
\mathfrak m_1 R_2$.  It follows that $\ejump_{k_2/k_1}(R_1) = 0$.\\

\noindent \textbf{Case:} 
$[K_1:K] = p$.\\
Since $K_1$ is the residue field of $K \otimes_k k_1$ and $K_1 \neq K$,
Lemma~\ref{lemma-roots-of-t} implies that 
$\sqrt[p]{t} \notin K$ and hence $\sqrt[p^2]{t} \notin K$.  
This 
means that 
$K_2$, which contains $\sqrt[p^2]{t}$ and is at most a
$p^2$-dimensional vector space over $K$, 
must be precisely $K_2 = K(\sqrt[p^2]{t})$ with $[K_2:K_1] = p$. 
It follows that
$\mathfrak m_2 = \mathfrak m_1 R_2$, and hence that
$\ejump_{k_2/k_1}(R_1) = 0$.\\

\noindent\textbf{Case:}
$[K_1:K] = [K_2:K_1] = 1$.\\
Since $K = K_2$ is the residue field of $K \otimes_k k_2$,
Lemma~\ref{lemma-roots-of-t} implies that
$\sqrt[p^2]{t} \in K$.  Another application of
Lemma~\ref{lemma-roots-of-t} yields the isomorphisms
\begin{align}
K \otimes_k k_1 & \cong R_1/\mathfrak m \cong
  K[\eps_1]/(\eps_1^p),\\
K \otimes_k k_2 & \cong R_2/\mathfrak m 
  \cong K[\eps_2]/(\eps_2^{p^2}),
  \end{align}
  where the natural inclusion $R_1/\mathfrak m
  \to R_2/\mathfrak m$ is
  given by $\eps_1 \mapsto \eps_2^p$.
  Choosing $f \in \mathfrak m_2$ to be
  any lift of $\eps_2$, it follows that $f^p \in
  \mathfrak m_1$ is a lift of $\eps_1$.  
  Therefore $\mathfrak m_2 =
  \mathfrak m R_2 + (f)$ and $\mathfrak m_1 = \mathfrak m R_1 +
  (f^p)$.  
  Notice that $f^p \notin \mathfrak m_1^2$ and so by
  Nakayama's lemma, it is included in 
  some minimal set of generators $f^p, x_2, x_3,\ldots,x_m$
  for the $R_1$-ideal $\mathfrak m_1$.  Furthermore, we may choose
  these generators so 
  that $x_2,\ldots, x_m \in 
  \mathfrak m R_1$.  Here $m = \edim(R_1) = \edim(R)$, since
  $\ejump_{k_1/k}(R) = 0$.
  As ideals in $R_2$, we have 
\begin{align*}
(f,x_2,\ldots, x_m) &= (f) + (f^p,x_2,\ldots, x_m)\\
  & = (f) + \mathfrak m_1 R_2\\
  & = (f) + (f^p) + \mathfrak m R_2\\
  & = \mathfrak m_2.
\end{align*}
Therefore $\edim(R) \leq \edim(R_2) \leq m = \edim(R)$, 
and hence $\ejump_{k_2/k}(R) = 0$.
\end{proof}
\begin{corollary}\label{corollary-embedding-jump-height-one}
  Let $R$ be a local noetherian $k$-algebra.  
  If $K/k$ is a purely inseparable field extension
  and $K' := 
  K \cap k^{1/p}$, then 
  $$\ejump_{K/k}(R) = \ejump_{K'/k}(R).$$
  \end{corollary}
\begin{proof}
  Since $R$ is noetherian, we may assume that
  $K$ is finitely generated, so that $K =
  k(\sqrt[p^{n_1}]{t_1},\ldots, \sqrt[p^{n_r}]{t_r})$ 
  for certain $t_i \in k\setminus k^p$.
  It follows by inducting on $r$ and applying
  Proposition~\ref{proposition-embedding-jump} that
  $\ejump_{K/k}(R) =  \ejump_{K'/k}(R)$, where
  $K' := k(\sqrt[p]{t_1}, \ldots, \sqrt[p]{t_r}) = K \cap k^{1/p}$.  
\end{proof}

We recover, as a further corollary,  the following result
(cf.~\cite[Thm.~IV.0.22.5.8]{EGA}): 

\begin{theorem}[EGA]
Let $X$ be a regular variety over a field $k$. 
  $X$ is smooth over $k$ if and only if $X \times_k k^{1/p}$ is
  a regular variety over $k^{1/p}$.
\end{theorem}

\begin{proof}
  Let $R = \OO_{X,x}$, for an arbitrary point $x \in X$.
  It follows from Corollary~\ref{corollary-embedding-jump-height-one}
  that 
  $\ejump_{k^{1/p^\infty}/k}(R) = \ejump_{k^{1/p}/k}(R).$
  Since
  $\bar k/ k^{1/p^\infty}$ is a separable field extension, 
  $R \otimes_k \bar k$ is regular if and only if $R
  \otimes_k k^{1/p^\infty}$ is regular.
  The result then follows from the observation that for any field
  extension $k'/k$ and regular local ring $R$, the base change
  $R \otimes_k k'$ is 
  regular if and only if
  $\ejump_{k'/k}(R) = 0$.
\end{proof}

\section{Future directions}
\setcounter{equation}{0}
We leave as an open question for future research:

\begin{question}
  Does there exist a regular del Pezzo surface $X$ with $H^1(X,\OO_X)
  \neq 0$ over a field of transcendence degree $1$ over a perfect
  field? 
\end{question}

By Proposition~\ref{proposition-h1=0-in-p>3}, if an example does
exist, it occurs in characteristic $2$ or $3$.
The author has constructed examples in characteristic $2$ of regular
del Pezzo surfaces $X$ with $H^1(X,\OO_X) \neq 0$ over fields of
transcendence degree at least $3$ 
 (cf.~\cite{maddock-del}) and shall describe a similar example
over a field of transcendence degree $2$ in a
forthcoming paper.

%
%

\bibliographystyle{alpha}
\bibliography{refs}

\end{document}